\documentclass[12pt,a4paper]{amsart}
\usepackage{graphicx}
\usepackage{amssymb}
\usepackage{amsfonts}
\usepackage{amsmath}
\usepackage{array}
\usepackage{rotating}
\usepackage[matrix,frame]{xypic}

\usepackage{tikz}
\usetikzlibrary{positioning,arrows,matrix,chains,shadows,shapes,calc,trees}
\tikzstyle{Point} = [fill, radius=0.08]
\tikzset{baseline={([yshift=-3.5pt]current bounding box.center)}}
\tikzset{level distance = 0.7cm, sibling distance = 2em}
\tikzset{edge from parent/.style={
       draw,
       edge from parent path = {(\tikzparentnode) --
                                (\tikzchildnode)}
    }}
\tikzset{root/.style={}}
\tikzset{white/.style={}}

\def\arbsu#1#2#3#4#5#6{
\xymatrix@R=0.1cm@C=2mm{
 & & & & {\GrTeXBox{#1}}\arx1[llld]\arx1[rrrd]\\
& {\GrTeXBox{#2}}\arx1[dl]\arx1[dr] &&&&&&
  {\GrTeXBox{#3}}\arx1[dl]\arx1[dr] \\
{\GrTeXBox{}} && {\GrTeXBox{#4}}\arx1[dl]\arx1[dr]
  &&&& {\GrTeXBox{#5}}\arx1[dl]\arx1[dr] && {\GrTeXBox{}}\\
& {\GrTeXBox{#6}}\arx1[dl]\arx1[dr] && {\GrTeXBox{}} && {\GrTeXBox{}}
&& {\GrTeXBox{}} \\
{\GrTeXBox{}} && {\GrTeXBox{}} \\
}
}

\def\arbsd#1#2#3#4#5{
\newdimen\vcadre\vcadre=0.01cm 
\newdimen\hcadre\hcadre=0.01cm 
\xymatrix@R=0.1cm@C=2mm{
&&& {\GrTeXBox{#1}}\arx1[llld]\arx1[rrrd]\arx1[d]\\
{\GrTeXBox{}} &&&
  {\GrTeXBox{#2}}\arx1[dl]\arx1[dr] &&&
  {\GrTeXBox{#3}}\arx1[dl]\arx1[dr] \\
&& {\GrTeXBox{#4}}\arx1[dl]\arx1[dr]
  && {\GrTeXBox{}} & {\GrTeXBox{#5}}\arx1[dl]\arx1[dr] && {\GrTeXBox{}}\\
& {\GrTeXBox{}} && {\GrTeXBox{}} & {\GrTeXBox{}} && {\GrTeXBox{}} \\
}
}

\def\arbst#1#2#3#4#5{
\newdimen\vcadre\vcadre=0.01cm 
\newdimen\hcadre\hcadre=0.01cm 
\xymatrix@R=0.1cm@C=2mm{
&&&& {\GrTeXBox{#1}}\arx1[llld]\arx1[rrrd] \\
& {\GrTeXBox{#2}}\arx1[dl]\arx1[dr] &&&&&&
  {\GrTeXBox{#3}}\arx1[dl]\arx1[dr] \\
{\GrTeXBox{}} && {\GrTeXBox{#4}}\arx1[dl]\arx1[dr]\arx1[d]
  &&&& {\GrTeXBox{#5}}\arx1[dl]\arx1[dr] && {\GrTeXBox{}} \\
& {\GrTeXBox{}} & {\GrTeXBox{}} & {\GrTeXBox{}} &&& {\GrTeXBox{}}
&& {\GrTeXBox{}} \\
}
}

\def\arbsq#1#2#3#4#5{
\newdimen\vcadre\vcadre=0.01cm 
\newdimen\hcadre\hcadre=0.01cm 
\xymatrix@R=0.1cm@C=2mm{
 & & & & {\GrTeXBox{#1}}\arx1[llld]\arx1[rrrd]\\
& {\GrTeXBox{#2}}\arx1[dl]\arx1[dr] &&&&&&
  {\GrTeXBox{#3}}\arx1[dl]\arx1[d]\arx1[dr] \\
{\GrTeXBox{}} && {\GrTeXBox{#4}}\arx1[dl]\arx1[dr]
  &&&& {\GrTeXBox{}} & {\GrTeXBox{}} & {\GrTeXBox{}}\\
& {\GrTeXBox{#5}}\arx1[dl]\arx1[dr] &&\\
{\GrTeXBox{}} && {\GrTeXBox{}} \\
}
}

\def\arbsc#1#2#3#4{
\newdimen\vcadre\vcadre=0.01cm 
\newdimen\hcadre\hcadre=0.01cm 
\xymatrix@R=0.1cm@C=2mm{
&&& {\GrTeXBox{#1}}\arx1[llld]\arx1[rrrd]\arx1[d]\\
{\GrTeXBox{}} &&&
  {\GrTeXBox{#2}}\arx1[dl]\arx1[dr]\arx1[d] &&&
  {\GrTeXBox{#3}}\arx1[dl]\arx1[dr] \\
&& {\GrTeXBox{}} & {\GrTeXBox{}} & {\GrTeXBox{}}
 & {\GrTeXBox{#4}}\arx1[dl]\arx1[dr] && {\GrTeXBox{}}\\
&&&& {\GrTeXBox{}} && {\GrTeXBox{}} \\
}
}

\def\arbss#1#2#3#4{
\newdimen\vcadre\vcadre=0.01cm 
\newdimen\hcadre\hcadre=0.01cm 
\xymatrix@R=0.1cm@C=2mm{
&&& {\GrTeXBox{#1}}\arx1[llld]\arx1[rrrd]\arx1[d]\\
{\GrTeXBox{}} &&&
  {\GrTeXBox{#2}}\arx1[dl]\arx1[dr] &&&
  {\GrTeXBox{#3}}\arx1[dl]\arx1[d]\arx1[dr] \\
&& {\GrTeXBox{#4}}\arx1[dl]\arx1[dr]
  && {\GrTeXBox{}} & {\GrTeXBox{}} && {\GrTeXBox{}}\\
& {\GrTeXBox{}} && {\GrTeXBox{}} & {\GrTeXBox{}} && {\GrTeXBox{}} \\
}
}

\def\arbsse#1#2#3#4{
\newdimen\vcadre\vcadre=0.01cm 
\newdimen\hcadre\hcadre=0.01cm 
\xymatrix@R=0.1cm@C=2mm{
&&&& {\GrTeXBox{#1}}\arx1[llld]\arx1[rrrd] \\
& {\GrTeXBox{#2}}\arx1[dl]\arx1[dr] &&&&&&
  {\GrTeXBox{#3}}\arx1[dl]\arx1[d]\arx1[dr] \\
{\GrTeXBox{}} && {\GrTeXBox{#4}}\arx1[dl]\arx1[dr]\arx1[d]
  &&&& {\GrTeXBox{}} && {\GrTeXBox{}} \\
& {\GrTeXBox{}} & {\GrTeXBox{}} & {\GrTeXBox{}} &&& {\GrTeXBox{}}
&& {\GrTeXBox{}} \\
}
}

\def\arbsh#1#2#3{
\newdimen\vcadre\vcadre=0.01cm 
\newdimen\hcadre\hcadre=0.01cm 
\xymatrix@R=0.1cm@C=2mm{
&&& {\GrTeXBox{#1}}\arx1[llld]\arx1[rrrd]\arx1[d]\\
{\GrTeXBox{}} &&&
  {\GrTeXBox{#2}}\arx1[dl]\arx1[dr]\arx1[d] &&&
  {\GrTeXBox{#3}}\arx1[dl]\arx1[dr]\arx1[d] \\
&& {\GrTeXBox{}} & {\GrTeXBox{}} & {\GrTeXBox{}}
 & {\GrTeXBox{}} & {\GrTeXBox{}} & {\GrTeXBox{}}\\
}
}

\newenvironment{arb}{\begin{tikzpicture}[
scale=0.5,level distance=7mm,level 1/.style={sibling distance=8mm},level 2/.style={sibling distance=5mm},level 3/.style={sibling distance=5mm},grow'=down, font=\scriptsize]
\tikzstyle{ve}=[draw,circle,inner sep=1pt,fill] 
\tikzstyle{vv}=[draw,circle,inner sep=1pt] 
\tikzstyle{vee}=[minimum size=0pt ,inner sep=0pt]}{\end{tikzpicture}}

\newcommand{\vv}{node[vv] {}}
\newcommand{\ve}{node[ve] {}}

\def\arbA{{\scriptstyle \circ}}

\def\arbB{\begin{arb}
\node[ve] {}
child{\vv} child{\vv};
\end{arb}}

\def\arbCA{\begin{arb}
\node[ve] {}
child{\ve  child{\vv} child{\vv}} child{\vv};
\end{arb}}

\def\arbF{\begin{arb}
\node[ve] {}
child{\ve  child{\vv} child{\vv}} child{\ve  child{\vv} child{\ve  child{\vv} child{\vv} child{\vv} child{\vv}}} child{\vv};
\end{arb}}

\newtheorem{example}{Example}[section]
\newtheorem{note}[example]{Note}
\newtheorem{theorem}[example]{Theorem}

\newtheorem{definition}[example]{Definition}
\newtheorem{proposition}[example]{Proposition}

\newtheorem{lemma}[example]{Lemma}

\def\Proof{\noindent \it Proof -- \rm}
\def\qed{\hspace{3.5mm} \hfill \vbox{\hrule height 3pt depth 2 pt width 2mm}
\bigskip}
\def\K{{\mathbb K}}

\def\wt{{\rm wt}}
\def\SS{{\mathcal{S}}}
\def\PT{{\rm PT}}

\def\NC{{\rm NC}}

\def\<{\langle}
\def\>{\rangle}
\def\R{{\mathbb R}}
\def\C{\operatorname{\mathbb C}}

\def\S{{\bf S}}

\def\SG{{\mathfrak S}}
\def\H{{\mathcal H}}

\def\Sym{{\bf Sym}}

\def\Hnc{{\mathcal H}_{\rm ncdif}}
\def\NDPF{{\rm NDPF}}
\def\ev{{\rm ev}}
\def\ii{{\rm int}}

\def\Res{\operatorname{Res}}

\def\cut{\operatorname{cut}}
\def\Res{\operatorname{Res}}

\def\PST{{\rm PST}} 

\def\Tabvrule{\vrule width-0.4pt}       
\def\Tabhrule{\hrule \hrule height-0.4pt} 
\def\Tabstrut{\vrule height2.2ex 
                     depth0.8ex  
                     width0ex    
\relax}

\def\PasCase#1{\omit%
            $\vcenter{\hbox {\vbox to 0.4pt{}}
               \hbox{\makebox[3ex]{\Tabstrut$#1$}}}%
               \Tabvrule$}
\def\PasCasePoint{\PasCase{\cdot}}
\def\DessinCarre#1{%
    \vcenter{\hbox{}\hrule
             \hbox{\vrule\makebox[3ex]{\Tabstrut$#1$}\vrule}\Tabhrule}%
             \Tabvrule}
\def\GenRuban#1{\vcenter{\halign{&$\DessinCarre{##}$\cr#1}}\egroup}

\def\sTabvrule{\vrule width-0.4pt}
\def\sTabhrule{\hrule \hrule height-0.4pt}
\def\sTabstrut{\vrule height1.6ex depth0.6ex width0ex \relax}
\def\sDessinCarre#1{%
    \vcenter{\hbox{}\hrule
             \hbox{\vrule\makebox[2.3ex]%
                  {\sTabstrut$\scriptstyle#1$}\vrule}\sTabhrule}%
             \sTabvrule}
\def\sGenRuban#1{\vcenter{\halign{&$\sDessinCarre{##}$\cr#1}}\egroup}

\def\ruban{%
  \bgroup
  \let\ =\omit
  \let\\=\cr
  \let\x=\times
  \let\.=\PasCasePoint
  \offinterlineskip
  \GenRuban}

\def\sruban{%
  \bgroup
  \let\ =\omit
  \let\x=\times
  \let\\=\cr
  \offinterlineskip
  \sGenRuban}

\title[]%
{Free cumulants, Schr\"oder trees, and operads}

\author[M. Josuat-Verg\`es, F. Menous, J.-C.~Novelli and J.-Y.~Thibon]%
{Matthieu Josuat-Verg\`es, Fr\'ed\'eric Menous,\\ Jean-Christophe Novelli, and Jean-Yves Thibon}

\address[Menous]{Laboratoire de Math\'ematiques d'Orsay,\\ Univ.
Paris-Sud,\\ CNRS,\\ Universit\'e Paris Saclay,\\ B\^atiment 425,\\ 91405 Orsay Cedex,\\ FRANCE}

\address[Josuat-Verg\`es, Novelli, and Thibon] {Laboratoire d'Informatique Gaspard Monge,\\ Universit\'e Paris-Est Marne-la-Vall\'ee, \\
5 Boulevard Descartes, \\Champs-sur-Marne, \\77454 Marne-la-Vall\'ee cedex 2 \\
FRANCE}
\email[Matthieu Josuat-Verg\`es]{josuatv@univ-mlv.fr}
\email[Fr\'ed\'eric Menous]{frederic.menous@math.u-psud.fr}
\email[Jean-Christophe Novelli]{novelli@univ-mlv.fr}
\email[Jean-Yves Thibon]{jyt@univ-mlv.fr} 
\date{}

\begin{document}

\begin{abstract}
The functional equation defining the free cumulants in  free probability 
is lifted 
successively to the noncommutative Fa\`a di Bruno algebra, and then 
to the group of a free operad over Schr\"oder trees. This leads to new
combinatorial expressions, which remain valid for operator-valued
free probability. Specializations of these expressions give back 
Speicher's formula in terms of noncrossing partitions, and its interpretation
in terms of characters due to Ebrahimi-Fard and Patras.
\end{abstract}

\maketitle

\section{Introduction}

\subsection{Functional equations and combinatorial Hopf algebras}
Recent works on certain functional equations
involving reversion of formal power series  have revealed
that the appropriate setting for their combinatorial understanding  
involved a series of noncommutative generalizations, ending up as
an equation in the group of an operad.

Roughly speaking, this amounts to first interpreting the equation in the 
Fa\`a di Bruno Hopf algebra, lifting it to its noncommutative version, and
then replacing the constant term by a new indeterminate, giving rise to tree 
expanded series. 

This approach to functional inversion has been initiated  in \cite{NTPark,NTLag,NTDup},
and then extended in \cite{MNT} to deal with the conjugacy equation for formal
diffeomorphisms.

Free probability provides other examples of functional equations with a combinatorial solution. 
The relation between the moments and the free cumulants of a single random variable
is just a functional inversion, which can be treated combinatorially by the
formalism of \cite{MNT}. However, the case of several random variables is classically
formulated as a triangular system of equations which is solved by M\"obius inversion
over the lattice of noncrossing partitions \cite{Spei1,Spei2}.

We shall see that this system can be encoded by a single equation in the group of an operad.
This version encompasses the case of an operator valued probability. The solution
arises as a sum over reduced plane trees which reduces to Speicher's solution in the scalar case.
Also, our functional equation gives back that of Ebrahimi-Fard and Patras \cite{EFP}, which
interpret the series of the moments as a character of a Hopf algebra, and that of the free cumulants
as an infinitesimal character, both related by a dendriform exponential.
We have here a similar structure.  

\medskip
{\footnotesize
{\it Acknowledgements.-}
This research has been partially supported by
the project CARMA of the French Agence Nationale
de la Recherche.
}

\subsection{Classical probability}

The free cumulants $k_n$ of a probability measure $\mu$ on $\R$
are defined (see {\it e.g.,}~\cite{Spei1}) by means of the generating series
of its moments $m_n$
\begin{equation}
M_\mu(z) :=\int_\R\frac{\mu(dx)}{z-x}=z^{-1}+\sum_{n\ge 1}m_nz^{-n-1}
\end{equation}
as the coefficients of its compositional inverse 
\begin{equation}
K_\mu(z) :=M_\mu(z)^{\<-1\>}=z^{-1}+\sum_{n\ge 1}k_n z^{n-1}\,.
\end{equation}
 The formal series $M_\mu$ is called the Cauchy transform of $\mu$,
and $K_\mu$ its ${\mathcal R}$-transform. By abuse of language, we shall also say that
$K_\mu$ is the ${\mathcal R}$-transform of $M_\mu$.

It is in general instructive to interpret the coefficients of
a formal power series as the specializations of the elements of some 
generating family of the algebra of symmetric functions.

The classical algebra of symmetric functions, denoted by $Sym$ or $Sym(X)$, is
a free associative and commutative graded algebra with one generator in each degree:
\begin{equation}
Sym = \C[h_1,h_2,\ldots] = \C[e_1,e_2,\ldots] = \C[p_1,p_2,\ldots]
\end{equation}
where the $h_n$ are the complete homogeneous symmetric functions, the $e_n$ the elementary
symmetric functions, and the $p_n$ the power sums. 

In this context, it is the interpretation (in the notation of \cite{Mcd})
\begin{equation}
m_n=\phi(h_n)=h_n(X)
\end{equation}
which is relevant.
Indeed, the process of functional inversion (Lagrange
inversion) admits a simple expression within this formalism
(see~\cite{Mcd}, ex.~24 p.~35). If the symmetric functions
$h_n^*$ are defined by the equations
\begin{equation}
u=tH(t) \ \Longleftrightarrow \ t=uH^*(u)
\end{equation}
where $H(t):=\sum_{n\ge 0}h_nt^n$, $H^*(u):=\sum_{n\ge 0}h_n^*u^n$,
then, 
\begin{equation}
h_n^*(X)==\frac{1}{n+1} [t^n] E(-t)^{n+1}
\end{equation}
where $E(t)$ is defined by $E(t)H(t)=1$.
This defines an involution $f\mapsto f^*$ of the ring of symmetric
functions.
Now, if one sets $m_n=h_n(X)$ as above, then
$M_\mu(z)=z^{-1}H(z^{-1})=u$, so that 
\begin{equation}
z= K_\mu(u)=\frac{1}{u}E^*(-u) =
u^{-1}+\sum_{n\ge 1}(-1)^ne_n^* u^{n-1}\,,
\end{equation}
and finally
\begin{equation}  
k_n=(-1)^ne_n^*(X)\,.
\end{equation}
It follows  from the explicit formula (see~\cite{Mcd} p. 35)
\begin{equation}
-e_n^*=\frac1{n-1}\sum_{\lambda\vdash n}
\binom{n-1}{l(\lambda)} \binom{l(\lambda)}{m_1, m_2,\ldots,m_n} e_\lambda
\end{equation}
(where $\lambda=1^{m_1}2^{m_2}\cdots n^{m_n}$) that $-e_n^*$
is Schur positive, and moreover that $-e_n^*$ is the Frobenius characteristic
of a permutation representation $\Pi_n$, twisted by the sign character.
Let 
\begin{equation}
(-1)^{(n-1)}k_n = -e_n^*=:\omega(f_n)
\end{equation}
so that $f_n$ is the character of $\Pi_n$.
It is proved in \cite{NTPark} that $f_n$ is the characteristic
of the action of $\SG_n$ on prime parking functions.
These considerations can be extended to noncommutative symmetric functions,
the symmetric group being replaced by the $0$-Hecke algebra \cite{NTLag}.

\subsection{Free probability}

A {\it noncommutative probability space} is a pair $(A,\phi)$ where $A$ is a
unital algebra over $\C$ and $\phi$ a linear form on $A$ such that
$\phi(1)=1$ (see, {\it e.g.}, \cite{Spei1,Spei2}).

The {\it free moments} are the functions $m_n$ defined by
\begin{equation}
m_n[a_1,\ldots,a_n]=\phi(a_1\cdots a_n).
\end{equation}

The {\it free cumulants} $\kappa_n$ are defined by the implicit equations
\begin{equation}
\phi(a_1\cdots a_n)=\sum_{\pi\in\NC_n}\kappa_\pi[a_1,\ldots,a_n]
\end{equation}
where $\NC_n$ is the set of noncrossing partitions of $[n]$ and 
\begin{equation}
k_\pi[a_1,\ldots,a_n]=\prod_{B\in\pi}\kappa[B] \ \text{and for $B=\{b_1<\ldots < b_p\}$,}\ \kappa[B]=\kappa[b_1,\ldots,b_p]. 
\end{equation}
By M\"obius inversion over the lattice of noncrossing partitions, this yields
\begin{equation}  \label{kappaphi}
\kappa_\pi [a_1\cdots a_n]=\sum_{\sigma\le \pi}  \mu(\sigma,\pi)  \phi_\sigma[a_1,\ldots,a_n].
\end{equation}
where $\phi_\pi$ is defined similarly \cite{Spei1}.

\section{The Fa\`a di Bruno Hopf algebra and its noncommutative analogue}

\subsection{The Fa\`a di Bruno algebra and symmetric functions}

The algebra of symmetric functions is also a Hopf algebra.
Its usual bialgebra structure is defined by the coproduct
\begin{equation}
\Delta_0 h_n = \sum_{i=0}^nh_i\otimes h_{n-i}\quad (h_0=1)
\end{equation}
which allows to interpret it as the algebra of polynomial functions on the multiplicative group
\begin{equation}
G_0 = \{a(z)=\sum_{n\ge 0}a_n z^n\ (a_0=1)\}
\end{equation}
of formal power series with constant term 1: $h_n$ is the coordinate function 
\begin{equation}
h_n:\ a(z)\longmapsto a_n.
\end{equation}
Indeed, $h_n(a(z)b(z))=\Delta_0 h_n (a(z)\otimes b(z))$.

But $h_n$ can also be interpreted as a coordinate on the group
\begin{equation}
G_1 = \{A(z)=\sum_{n\ge 0}a_n z^{n+1}\ (a_0=1)\}
\end{equation}
of formal diffeomorphisms tangent to identity, under functional composition.
Again with $h_n(A(z))=a_n$ and $h_n(A(z)B(z))=\Delta_1(A(z)\otimes B(z))$,
the coproduct is now
\begin{equation}
\Delta_1 h_n = \sum_{i=0}^nh_i\otimes h_{n-i}((i+1)X)\quad (h_0=1)
\end{equation}
where $h_n(mX)$ is defined as the coefficient of $t^n$ in $(\sum h_kt^k)^m$.
The resulting bialgebra is known as the Fa\`a di Bruno algebra \cite{JR}.
By definition, its antipode sends $H(t)$ to its functional inverse.

\subsection{Noncommutative symmetric functions and noncommutative formal diffeomorphisms}

These constructions can be repeated word for word with the algebra
$\Sym$ of noncommutative symmetric functions. It is a free associative
(and noncommutative) graded algebra with one generator $S_n$ in each degree,
which can be interpreted as above if the coefficients $a_n$ belong to
a noncommutative algebra. In this case, $G_0$ is still a group, but $G_1$
is not, as composition is not anymore associative. However, the coproduct $\Delta_1$
\begin{equation}
\Delta_1 S_n = \sum_{i=0}^nS_i\otimes S_{n-i}((i+1)A)\quad (S_0=1)
\end{equation}
remains coassociative, and $\Sym$ endowed with this coproduct is a Hopf algebra,
known as Noncommutative Formal Diffeomorphims \cite{BFK,NTLag}, or as the noncommutative
Fa\`a di Bruno algebra \cite{EFLM}.

Let $\Hnc$ denote this Hopf algebra, and let $\gamma$ denote its antipode.
The image $h=\gamma(\sigma_1)$ 
of the formal sum of its generators
\begin{equation}
\sigma_1 = \sum_{n\ge 0}S_n
\end{equation}
is characterized by the functional equation
\begin{equation} 
h^{-1} = \sum_{n\ge 0}S_n h^n.
\end{equation}

The {\it noncommutative Lagrange series} $g$
is defined by the functional equation
\begin{equation}\label{eq:lag} 
g = \sum_{n\ge 0}S_n g^n
\end{equation}
Recall that for $f\in\Sym$, $f(-A)$ is the image
of $f$ by the automorphism $S_n\mapsto (-1)^n\Lambda_n$.
It is proved in \cite{NTLag} that $h(A)=g(-A)$, and that
\begin{equation}\label{eq:devlag}
g_n = \sum_{\pi\in\NDPF_n}S^{\ev(\pi)}
\end{equation}
where $\NDPF_n$ is the set of nondecreasing parking functions of length $n$,
{\it i.e.}, nondecreasing words over the positive integers such that $\pi_i\le i$,
and $\ev(\pi)=(|\pi|_i)_{i=1..n}$. The first terms are
\begin{equation}
\label{g01234}
\begin{split}
& g_0 =1,\qquad g_1 = S_1,\qquad g_2 = S_2 + S^{11}\,, \\
& g_3 = S_3 + 2S^{21} + S^{12} + S^{111}\,, \\
& g_4 = S_4 + 3S^{31} + 2S^{22} + S^{13} + 3S^{211} + 2S^{121} + S^{112} +
S^{1111}\,.
\end{split}
\end{equation}

There is a simple bijection between $\NDPF_n$ and $\NC_n$, and $g_n$ can as well
be written as a sum over noncrossing partitions.

\section{Noncommutative free cumulants}

In the case of a single random variable, the free cumulants $\kappa_n$ are the images 
of the noncommutative symmetric functions $K_n$ defined by the functional equation
\begin{equation}\label{ncfcum}
\sigma_1 = \sum_{n\ge 0}K_n \sigma_1^n
\end{equation}
by the character $\chi$ of $\Hnc$ such that $\chi(S_n)=m_n=\phi(a^n)$, where $a$
is some element of a noncommutative probability space $(A,\phi)$.

This equation is formally similar to \eqref{eq:lag}, so that
we can  write down immediately an expression of $S_n$ in terms
of the basis $K^I:=K_{i_1}\cdots K_{i_r}$, by replacing $S^I$ by
$K^I$ in the expression of $g_n$ given in \eqref{eq:devlag}:
\begin{align}
S_1 &= K_1\\
S_2 &= K_2+K^{11}\\
S_3 &= K_3+2K^{21}+K^{12}+K^{111} \\
S_4 &= K_4+3K^{31}+2K^{22}+3K^{211}+K^{13}+2K^{121}+K^{112}+K^{1111}.
\end{align}
These expressions are sums over Catalan sets 
\begin{equation}
S_n =\sum_{\pi\in\NDPF_n}K^{\ev(\pi)}
\end{equation}
in the guise of nondecreasing parking functions, instead of noncrossing
partitions. 

Solving recursively for $K_n$, we find, in various bases of $\Sym$
\begin{align}
K_1 &= S_1 = \Lambda_1\\
K_2 &= S_2-S^{11} =- \Lambda_2=-R_{11}\\
K_3 &= S_3-2S^{21}-S^{12}+2S^{111} = \Lambda_3+\Lambda^{21}=R_{12}+2R_{111}\\
K_4 &= 
S_4-S^{13}-3(S^{31}-S^{121})-2(S^{22}-S^{112})+5(S^{211}-S^{1111})\\
    &= -\Lambda_4 -2\Lambda^{31}-\Lambda^{22}-\Lambda^{211} 
= -(R_{13}+2R_{121}+3R_{112}+5R_{1111}).
\end{align}

On these examples, $(-1)^{n-1}K_n$ appears to be given by the following rule:
start from the expression of $g_{n-1}$ on the $S^I$, add $1$ to the first parts,
and replace $S$ with $\Lambda$. In other words,
\begin{equation}
K_n = -(\Omega g_{n-1})(-A)
\end{equation}
where $\Omega$ is the linear operator defined in \cite{NTPark}
\begin{equation}
\Omega S^{i_1,\ldots,ir} = S^{i_1+1,i_2,\ldots,i_r} \quad \text{and $\Omega(1)=S_1$.}
\end{equation}
It is proved in \cite{NTPark} that
\begin{equation}
g^{-1} = 1-\Omega g,
\end{equation}
and we have indeed
\begin{theorem}
\begin{equation}\label{eq:Kg}
K := 1+\sum_{n\ge 1}K_n = g^{-1}(-A). 
\end{equation}
\end{theorem}

\Proof Let $\gamma$ be the antipode of $\Hnc$. It is proved in \cite{NTPark}
that $\gamma(\sigma_1)=g(-A)$. Thus, \eqref{eq:Kg} is equivalent to
 \begin{equation}
K = \gamma(\lambda_{-1}).
\end{equation}
We have
\begin{multline}
\Delta_1(\lambda_{-1}) = (\Delta_1\sigma_1)^{-1} =\left(\sum_{k\ge 0}S_k\otimes \sigma_1^{k+1} \right)^{-1}\\
 =\left(\sum_{k\ge 0}S_k\otimes \sigma_1^{k} \right)^{-1}(1\otimes \sigma_1^{-1})
=\sum_{k\ge 0}(-1)^k\Lambda_k\otimes \sigma_1^{k-1}
\end{multline}
so that, by definition of an antipode,
\begin{equation}
1 = \sum_{k\ge 0}(-1)^k\gamma(\Lambda_k)\sigma_1^{k-1}
\end{equation}
and finally
\begin{equation}
\sigma_1 = \sum_{k\ge 0}(-1)^k\gamma(\Lambda_k)\sigma_1^{k}
\end{equation}
which implies the identification
 \begin{equation}
K_n = (-1)^n\gamma(\Lambda_n).
\end{equation}
\qed

This result implies an expression of $(-1)^{n-1}K_n$ on the ribbon basis:
start from the expression of $g_{n-1}$, {\it e.g.,} for $n=4$
\begin{equation}
g_3 = 5R_3+3R_{21}+2R_{12}+R_{111},
\end{equation}
replace each $R_I$ by $R_{\overline I^\sim}$ (mirror conjugate
composition)
\begin{equation}\label{eq:g3R}
g_3 \rightarrow  5R_{111}+3R_{12}+2R_{21}+R_{3},
\end{equation}
and insert $1$ at the beginning of each composition 
 \begin{equation}\label{eq:K4R}
-K_4=  5R_{1111}+3R_{112}+2R_{121}+R_{13}.
\end{equation}

On the $S$-basis, we have for example
\begin{equation}
K_4 = (S^4-S^{13}) -3(S^{31}-S^{121}) -2(S^{22}-S^{112})+5(S^{211}-S^{111})
\end{equation}
which is obtained from  \eqref{eq:g3R} or \eqref{eq:K4R}
by observing that
\begin{equation}
R_{1I} = S_1R_I-\Omega R_I.
\end{equation}

\section{Free cumulants in the Schr\"oder operad} 

\subsection{The Schr\"oder operad and its group}

To obtain a combinatorial expression for $K_n$, 
one can work in the group
of the Schr\"oder operad as in \cite[Section 10.2]{MNT}. This will 
cover the case of several random variables, hence imply
Speicher's formula with noncrossing partitions, as well as the case
of an arbitrary operator valued probability $\phi$ (see, {\it e.g.}, 
\cite{Dy,Spei2}).

Let $\PT_n$ be the set of reduced plane trees, {\it i.e.}, plane
trees for which any internal node has at least two descendants.
The {\it Schr\"oder operad} \cite{MNT} is the $\C$-vector space
\begin{equation}
\SS = \bigoplus_{n \ge 1}\SS_n,\quad \text{where}\ \SS_n =\C \PT_{n}
\end{equation}
endowed with the composition operations
\begin{equation}
\SS_n \otimes \SS_{k_1} \otimes \ldots
\otimes \SS_{k_n} 
\longrightarrow \SS_{k_1 + \ldots + k_n}  
\text{ ($n \geqslant 1$, $k_i \geqslant 1$)}
\end{equation}
which map the tensor product of trees $t_0 \otimes t_1 \otimes \ldots \otimes t_n$
to the tree $t_0 \circ (t_1, \ldots, t_n)$ obtained by replacing the leaves of
$t_0$, from left to right, by the trees $t_1, \ldots, t_n$.

The number of leaves of a tree $t$ will be called its degree $d(t)$, and  we
define the weight $\wt(t)$ of a tree as its degree minus 1.

We can represent trees by noncommutative monomials in indeterminates $S_n$ ($n\ge 0$),
by interpreting a node of arity $k$ as a $k$-ary operator denoted by $S_{k-1}$, and
writing the resulting expression in Polish notation. 

 For example,
\begin{equation}
S^{\arbF}=
S_2S_0S_1S_3S_0S_0S_0S_0S_0S_1S_0S_0=S^{201300000100}
\end{equation}
is of degree 8 and weight $7=2+1+3+1$.
The sequence of exponents $I$ is called a {\it Schr\"oder pseudocomposition}.
The sum of the components of $I$ is therefore the weight of the associated tree.
We shall indifferently use the notations $S^I$ or $S^t$.  
Let $\hat{\SS}$ be the completion of the vector space 
$\SS$ with respect to the weight $\wt(t)=d(t)-1$.
The group of the operad $\SS$ is defined as \cite{KM,Cha}
\begin{equation}
 G_{\SS} = \left\{ \arbA + \sum_{n \geqslant 2} p_n ,
   \hspace{1em} p_n \in \SS_n \right\} \subset \hat\SS
\end{equation}
endowed with the composition product 
\begin{equation}\label{eq:opcomp}
 p \circ q = q + \sum_{n \geqslant 2} p_n \circ \left(
   \underset{n}{\underbrace{q, \ldots, q}} \right) \in G_{\SS} 
\end{equation}
for $p = \arbA + \sum_{n \geqslant 2} p_n$ and $q \in G_{\SS}$. 

Elements of $G_{\SS}$  can be described by their coordinates
\begin{equation}
p = \sum_{t \in \text{PT}} p_t t\quad \text{and}\quad q
  = \sum_{t \in \text{PT}} q_t t \quad (\text{PT}=\cup_{n\geq 1} \text{PT}_n).
\end{equation}
(with $q_{\arbA} = p_{\arbA} = 1$) so that the coordinates of $r = p \circ q$
are given by
\begin{equation}
r_t = \sum_{ t = t_0 \circ (t_1, \ldots, t_n)} p_{t_0} q_{t_1} \ldots q_{t_n}
\end{equation}

This allows to consider the group $G_\SS$ as the group of characters of a graded Hopf algebra $H_\SS$.
It is the noncommutative polynomial algebra over reduced plane trees $\text{PT}$ (with unit $\arbA$), 
endowed with the coproduct given by {\it admissible cuts} (see \cite{MEF,MNT}): 
an admissible cut of a tree $T$ is a possibly empty subset of internal vertices $c=\lbrace i_1,\dots, i_k \rbrace$ 
such that along any path from the root to a leaf, there is at most one internal vertex in $c$. For any such cut, one defines 
\begin{itemize}
\item $P^c(T)=T_{i_1}\dots T_{i_k}$ as the product of the subtrees of $T$ having their root in $c$, 
ordered as in $T$ from the top and from left to right.
\item $R^c(T)$ as the trunk which remains after removing these trees.
\end{itemize} 
The coproduct of  $H_\SS$ is 
\begin{equation}
\Delta (T)=\sum_c R^c(T) \otimes P^c(T)
\end{equation}
so that $H_\SS$ is a graded Hopf algebra. For instance 
\begin{equation}
{\Delta} \left( \arbCA \right)
 = \arbA \otimes \arbCA + \arbB \otimes \arbB  +\arbCA \otimes \arbA.
\end{equation}

The bijection between $G_\SS$ and the group of characters on $H_\SS$ is obvious: 
since $H_\SS$ is a polynomial algebra, a character $\chi$ is entirely determined by its restriction 
to trees of positive weight, in other words, by its \textit{residue} in the sense of \cite{EFP1} 
(which can be considered as an infinitesimal character $\Res(\chi)$) 
and the values of this residue are given by the  coordinates in $G_\SS$. 

%
%
%
%
%
%
%

\subsection{Operadic free cumulants}


Consider the series of corollas
\begin{equation}
f_c = S_{0} + \sum_{n \geqslant 1} S_nS_0^{n+1} . 
\end{equation}
The inverse of $f_c$ in $G_\SS$ is, in terms of trees,
\begin{equation} 
g_c = \sum_{t \in \text{PT}} (- 1)^{i (t)} S^t 
\end{equation}
where $i(t)$ denotes the number of internal nodes of $t$.
Indeed, denoting by $\wedge(t_1,\ldots,t_n)$ the tree whose
subtrees of the root are $t_1,\ldots,t_n$,
\begin{align} 
     g_c & =  S_0 +\displaystyle \sum_{n \geqslant 1} 
            \sum_{ t = \bigwedge (t_1 \cdot \ldots \cdot t_{n + 1})\atop t_i \in \text{PT}} 
                    (- 1)^{i (t)} S^{\bigwedge (t_1 \cdot \ldots \cdot t_{n + 1})}\\
     & =  S_0 +\displaystyle  \sum_{n \geqslant 1} 
\sum_{t = \bigwedge (t_1 \cdot \ldots \cdot t_{n + 1})\atop t_i \in \text{PT}} 
(- 1)^{1 + i (t_1) + \ldots i (t_{n + 1})} S_n S^{t_1} \ldots S^{t_{n + 1}}\\
     & =  S_0 - \sum_{n \geqslant 1} S_n g_c^{n + 1}
\end{align}
so that 
\begin{equation}
S_0 = g_c + \sum_{n \geqslant 1} S_n g_c^{n + 1} = f_c \circ g_c.
\end{equation}

\begin{definition}
A Schr\"oder tree is {\it prime} if its rightmost subtree is a leaf.
We denote by $\PST_n$ the set of prime Schr\"oder trees of weight $n$.
\end{definition}

Prime Schr\"oder trees are counted by the large Schr\"oder numbers. Indeed,
if $s(x)=1+x+3x^2+11x^3+45x^4+\cdots$ is the generating series of Schr\"oder trees
of weight $n$, then, that of prime Schr\"oder trees is
\begin{equation}
p(x) = 1+\frac{xs(x)}{1-xs(x)}= \frac{4}{3-x+\sqrt {1-6\,x+{x}^{2}}} = 
1+x+2{x}^{2}+6{x}^{3}+22{x}^{4}+\cdots
\end{equation}

The series $g_c$, introduced in  \cite[Eq. (158)]{MNT}, projects onto the antipode
$g(-A)$ of $\Hnc$ by the map $S_0\mapsto 1$. 
Imitating the interpretation of the Lagrange series
given in  \cite[Eq. (164)]{MNT}  and exchanging the roles of $g_n,K_n$ and $S_n$ as above,
we obtain the following result.

\begin{theorem}\label{th:kap}
Define $\eta$ by
\begin{equation}\label{eq:67}
g_c = \eta\cdot S_0, 
\end{equation}\label{eq:68}
and $\kappa$ by
\begin{equation}
\kappa := \eta^{-1}\cdot S_0
\end{equation}
where the exponent $-1$ denotes here the multiplicative inverse. 
Then, the image of $\kappa$ by the algebra morphism $S_0\mapsto 1$
is the series $K$ of  $\Sym$.
In terms of trees,
\begin{equation}\label{eq:PST}
\kappa_n = \sum_{t\in\PST_n} (-1)^{i(t)-1}S^t.
\end{equation}
\end{theorem}
\begin{proof}
For an element $f$ of the Schr\"oder group $G_S$, write $f=\tilde fS_0$,
and for $f,g\in G_\SS$, define
\begin{equation}
f\dashv g=    (\tilde{f}\circ g)S_0 =S_0 +((\tilde{f}-1)\circ g)  S_0
\end{equation}
This is a partial composition:
if \begin{equation} 
g=\arbA + \sum_{n \geqslant 2} g_n\ \text{and}\  f=\arbA + \sum_{n \geqslant 2} f_n, 
\end{equation}
then
\begin{equation}
f\dashv g = \arbA +\sum_{n\geq 2} f_n(
   \underset{n-1}{\underbrace{g, \ldots, g}},\arbA )
\end{equation}
From \eqref{eq:67} and \eqref{eq:68}, we have
\begin{equation}
\tilde\kappa g_c = S_0\ \text{and}\ (\tilde\kappa\circ f_c)S_0 = f_c,
\end{equation} 
so that
\begin{equation}\label{eq:kappa}
f_c = \kappa\dashv f_c
\end{equation} 
which implies \eqref{eq:PST}. Indeed, plugging $f_c$ in this
expression, we get an alternating sum of trees obtained by grafting
corollas to leaves of prime Schr\"oder trees $t$ except to the rightmost one.
The sign of the resulting tree $t'$ is $(-1)^{i(t)-1}=(-1)^{i(t')-k-1}$ if $k$ is the
number of grafted corollas. Hence, each $t'$ which is not a corolla has coefficient
$(1-1)^n$ where $n$ is the number of its internal nodes whose all descendants are leaves.
\end{proof}

For example
\begin{eqnarray}\label{ex:kappa}
\kappa &=& [
1-(S^{10}+S^{200}-S^{1100}-S^{1010} + S^{3000} -S^{21000}-S^{20100}-S^{20010}-S^{12000}\nonumber\\
&-&S^{10200} +S^{111000}+S^{110100}+S^{101100}+S^{101010}+S^{110010})+\cdots) ]^{-1}S_0\nonumber\\
&=& S_0 +S^{100}+S^{2000}-S^{11000}  \nonumber\\
&&+ S^{30000}-S^{210000}-S^{201000}-S^{120000}+S^{111000}+S^{1101000}+\cdots
\end{eqnarray}

Our formula is  multiplicity-free, and the number of terms is given by the
large Schr\"oder numbers. 
This is also the sum of the absolute values
of the M\"obius function
of the lattice of noncrossing partitions.   We shall in the sequel give a combinatorial interpretation
of this coincidence, and explain how it allows to recover Speicher's formula.

On the example we shall see that
this expression is finer than Speicher's  formula 
as the term $2\phi(a_1)\phi(a_2)\phi(a_3)$ is now separated into two binary trees. Actually,
this will allow to take into account the case where $\phi$ is valued in a noncommutative
algebra.

\subsection{The operadic ${\mathcal R}$-transform}
Define the complementary operation $\vdash$ by
\begin{equation}
f\vdash g = \arbA +\sum_{n\geq 2} f_n(
   \underset{n-1}{\underbrace{\arbA, \ldots, \arbA}},g )
\end{equation}
so that
\begin{equation}\label{dashrel}
f\circ g = (f\dashv g)\vdash g.
\end{equation}
We can define the {\it operadic ${\mathcal R}$-transform} of a series $f$ by 
\begin{equation}
{\mathcal R}(f) = \left(f^{\circ-1}\right)^{\vdash -1}.
\end{equation}
Applying \eqref{dashrel}, we have
\begin{equation}
(f\dashv f^{\circ-1})\vdash f^{\circ-1} = f\circ  f^{\circ-1} = S_0,
\end{equation}
so that
\begin{equation}
f^{\circ-1} = (f\dashv f^{\circ -1})^{\vdash -1}
\end{equation}
and finally
\begin{equation}
{\mathcal R}(f) = (f^{\circ-1})^{\vdash -1} = f\dashv f^{\circ-1}.
\end{equation}

Moreover, we have obviously the so-called right dipterous relation
\begin{equation}
(f\dashv g)\dashv h = f\dashv (g\circ h).
\end{equation}
Applying it with $g=f^{\circ-1}$ and $h=f$, we obtain
\begin{equation}
(f\dashv f^{\circ-1})\dashv f = f\dashv S_0 = f,
\end{equation}
that is:

\begin{proposition}
For any $f\in G_\S$, the ${\mathcal R}$-transform $k$ of $f$ satisfies
\begin{equation}
f = k\dashv f.
\end{equation} 
\end{proposition}

\subsection{Inverse ${\mathcal R}$-transform of the series of corollas}

Let $h$ be defined by 
\begin{equation}
h = f_c\dashv h.
\end{equation} 

\begin{theorem}
Let us say that a Schr\"oder tree is \emph{left directed} is the rightmost subtree of each internal node is a leaf.
Let {\rm LDST} denote the set of such trees.
Then,
\begin{equation}
h = \sum_{t\in {\rm LDST}}S^t.
\end{equation}
\end{theorem}

\Proof By definition,
\begin{equation}
h = (f_c^{\dashv-1})^{\circ -1},
\end{equation}
which has been computed in \cite[Section 10.2]{MNT}

The connection between the operadic ${\mathcal R}$-transform and its classical
version will be made precise in Section \ref{sec:spei}.

\subsection{A dendriform structure arising from the Schr\"oder group}

%
%
%

Equation 
\eqref{eq:kappa}
seems to have the same structure as that of \cite{EFP1}, 
which involves dendriform (or shuffle) and codendriform (or unshuffle) algebras.

In terms of the Hopf algebra structure, instead of the convolution of characters, which on trees $T$ (of positive weight) reads as
\begin{equation}
(f \ast g)(T)=\pi \circ (f\otimes g) \circ \Delta (T)=\sum_{c=\lbrace i_1,\dots, i_k \rbrace} f(R^c(T))g(P^c(T))
\end{equation}
we get
\begin{equation}\label{eq:dash}
(f \dashv g)(T)=\pi \circ (f\otimes g) \circ \Delta^+_{\prec} (T)=\sum_{c=\lbrace i_1,\dots, i_k \rbrace}^{\prec} f(R^c(T))g(P^c(T))
\end{equation}
where the sum is restricted to admissible cuts such that the rightmost leaf of $T$ remains in $R^c(T)$ or, 
equivalently, such that the rightmost subtree in $P^c(T)$ does not contain the rightmost leaf of $T$.

Extending  to Schr\"oder trees the constructions of \cite{EFP1,EFP}, we can now state:
\begin{theorem} \label{th:unshundec}
For any tree $T$ of positive weight, let 
\begin{equation}
\Delta^+_{\prec} (T)= \sum_{c=\lbrace i_1,\dots, i_k \rbrace}^{\prec} R^c(T)\otimes P^c(T) \textrm{ and } \Delta^+_{\succ} (T)=\Delta(T)-\Delta^+_{\prec} (T)
\end{equation}
These maps can be extended to the augmentation ideal $H^+_\SS$ of $H_\SS$ by the rule:
\begin{eqnarray}
\Delta^+_{\prec} (T_1 T_2\dots T_s) &=& \Delta^+_{\prec} (T_1).\Delta(T_2\dots T_s) \\
\Delta^+_{\succ} (T_1 T_2\dots T_s) &=& \Delta^+_{\succ} (T_1).\Delta(T_2\dots T_s) 
\end{eqnarray}
so that $H_\SS$ is a codendriform bialgebra.
\end{theorem}

The last statement means that on $H^+_\SS$, if
\begin{eqnarray}
\Delta(a) &=&\bar{\Delta}(a)+1\otimes a  +a \otimes 1 \\
\Delta^+_{\prec}(a)  &=& \Delta_{\prec} (a) +a \otimes 1 \\
\Delta^+_{\succ} (a) &=& \Delta_{\succ} (a) + 1 \otimes a 
\end{eqnarray}
then
\begin{eqnarray}
(\Delta_{\prec}\otimes I)\circ \Delta_{\prec} &=& (I\otimes \bar{\Delta})\circ \Delta_{\prec} \\
(\Delta_{\succ}\otimes I)\circ \Delta_{\prec} &=& (I\otimes \Delta_{\prec})\circ \Delta_{\succ} \\
(\bar{\Delta} \otimes I)\circ \Delta_{\succ} &=& (I\otimes \Delta_{\succ})\circ \Delta_{\succ}
\end{eqnarray}

\begin{proof}
It is easy to see that, for a forest $F=T_1 \dots T_s$, $\Delta_{\prec}(F)$ is the sum of terms $R^c(F) \otimes P^c(T)$ 
over the non trivial ($c\not=\emptyset$) admissible cuts of $F$,
 such that the rightmost leaf of $T_1$ is in $R^c(T)$. 
In the same way, $\Delta_{\succ}(F)$ is the sum of terms $R^c(F) \otimes P^c(T)$ over the non trivial ($P^c(T)\not=1$) 
admissible cuts of $F$ such that the rightmost leaf of $T_1$ is in $P^c(T)$.

The reduced coproduct being coassociative,  
\begin{equation}
(\bar{\Delta} \otimes I) \circ \bar{\Delta} = (I\otimes \bar{\Delta})\circ \bar{\Delta}
\end{equation}
when applied to a forest $F=T_1 \dots T_s$, this yields a sum of terms $F^{(1)}\otimes F^{(2)}\otimes F^{(3)}$ 
obtained after two successive  nontrivial admissible cuts ($F^{(i)}\in H^+_\SS$). 

As $\bar{\Delta}=\Delta_{\prec}+\Delta_{\succ}$ on $H^+_\SS$,
\begin{equation}
(\bar{\Delta} \otimes I) \circ \bar{\Delta}=(\Delta_{\prec}\otimes I)\circ \Delta_{\prec}+(\Delta_{\succ}\otimes I)\circ \Delta_{\prec}+(\bar{\Delta} \otimes I)\circ \Delta_{\succ},
\end{equation}
for  $F=T_1 \dots T_s$, the sum of terms $F^{(1)}\otimes F^{(2)}\otimes F^{(3)}$ of $(\bar{\Delta} \otimes I) \circ \bar{\Delta}(F)$ splits into three parts:
\begin{itemize}
\item Terms $F^{(1)}\otimes F^{(2)}\otimes F^{(3)}$ contributing to $(\Delta_{\prec}\otimes I)\circ \Delta_{\prec}(F)$, which  are those such that the rightmost leaf of $T_1$ is in $F^{(1)}$.
\item Terms $F^{(1)}\otimes F^{(2)}\otimes F^{(3)}$ contributing to $(\Delta_{\succ}\otimes I)\circ \Delta_{\prec}(F)$, which are those such that the rightmost leaf of $T_1$ is in $F^{(2)}$.
\item Terms $F^{(1)}\otimes F^{(2)}\otimes F^{(3)}$ contributing to $(\bar{\Delta} \otimes I)\circ \Delta_{\succ}(F)$, which are those such that the rightmost leaf of $T_1$ is in $F^{(3)}$.
\end{itemize}
But this sum is also equal to $(I\otimes \bar{\Delta})\circ \bar{\Delta}(F)$, and the decomposition
\begin{equation}
(I\otimes \bar{\Delta})\circ \bar{\Delta}=(I\otimes \bar{\Delta})\circ \Delta_{\prec}+ (I\otimes \Delta_{\prec})\circ \Delta_{\succ}+ (I\otimes \Delta_{\succ})\circ \Delta_{\succ}
\end{equation}
must coincide with the previous one. This proves the relations 
\begin{eqnarray*}
(\Delta_{\prec}\otimes I)\circ \Delta_{\prec} &=& (I\otimes \bar{\Delta})\circ \Delta_{\prec} \\
(\Delta_{\succ}\otimes I)\circ \Delta_{\prec} &=& (I\otimes \Delta_{\prec})\circ \Delta_{\succ} \\
(\bar{\Delta} \otimes I)\circ \Delta_{\succ} &=& (I\otimes \Delta_{\succ})\circ \Delta_{\succ}
\end{eqnarray*}
\end{proof}

As in \cite{EFP1}, the space ${\rm Lin}(H_\SS, k)$, which  is a $\K$-algebra for the convolution product 
\begin{equation}
(f*g)=\pi \circ (f\otimes g) \circ \Delta 
\end{equation}
is also a dendrifrom algebra for left and right half-convolutions
\begin{eqnarray}
(f\prec g)&=&\pi \circ (f\otimes g) \circ \Delta_{\prec} \\
(f\succ g)&=&\pi \circ (f\otimes g) \circ \Delta_{\succ}. 
\end{eqnarray}
In terms of characters in $G_\SS$, the operation $\dashv$ (Equation (\ref{eq:dash})) coincides on trees with 
$\prec$, and the relation between $f_c$ and $\kappa$ in Theorem \ref{th:kap} translates now into the character equation
\begin{equation}
f_c=\epsilon +(\Res(\kappa))\prec f_c
\end{equation}
where $\epsilon$ is the unit of the group, corresponding to $S_0$.

As we shall see, this equation implies \cite[Th. 13]{EFP}, and thus also
Speicher's formula for the free cumulants. We shall first explain in the forthcoming
section how to derive the latter from  \eqref{eq:PST} by a direct combinatorial argument.

\section{Speicher's formula}\label{sec:spei}

Equation \eqref{eq:PST} is a formula for the free cumulants in terms of moments, involving
prime Schr\"oder trees instead of noncrossing partitions as in 
Equation~\eqref{kappaphi} (which was originally the definition of free 
cumulants). 

To recover a noncrossing partition from a tree, label the sectors
from left to right so as to obtain
the identity permutation by flattening the tree:
 

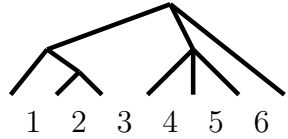
\[
   \tikzset{every picture/.style={scale=0.6}}
  \begin{tikzpicture}
      \tikzstyle{ver} = [circle, draw, fill, inner sep=0.5mm]
      \tikzstyle{edg} = [line width=0.6mm]
      \node      at (1.5,-0.6) {1};
      \node      at (2.5,-0.6) {2};
      \node      at (3.5,-0.6) {3};
      \node      at (4.5,-0.6) {4};
      \node      at (5.5,-0.6) {5};
      \node      at (6.5,-0.6) {6};
      \draw[edg] (2,0) to (2.5,0.5);
      \draw[edg] (3,0) to (2.5,0.5);
      \draw[edg] (2.5,0.5) to (1.8,1);
      \draw[edg] (1,0) to (1.8,1);
      \draw[edg] (4,0) to (5,1);
      \draw[edg] (5,0) to (5,1);
      \draw[edg] (6,0) to (5,1);
      \draw[edg] (1.8,1) to (4.5,2);
      \draw[edg] (5,1) to (4.5,2);
      \draw[edg] (7,0) to (4.5,2);
  \end{tikzpicture}
\]

The blocks of the partition are then formed by the adjacent sectors.
Here we obtain $1|2|36|45$.

But one can read more information. If we consider that $\phi$ is an arbitrary endomorphism of an associative
 algebra $A$ (instead of a linear form), then one reads the expression (compare \cite[Ex. 2.1.2]{Spei2})
 \begin{equation}
 \phi(\phi(a_1\phi(a_2))a_3\phi(a_4a_5)a_6).
 \end{equation}
Denote by $\kappa_n[a_1,\ldots,a_n]$ the evaluation of $\kappa_n$ obtained by this process.
For example (compare Eq. \eqref{ex:kappa}),
\begin{equation}
\begin{split}
\kappa_3[a_1a_2a_3]=&\phi(a_1a_2a_3)-\phi(\phi(a_1)a_2a_3)-\phi(a_1\phi(a_2)a_3) - \phi(\phi(a_1a_2)a_3)\\
&+  \phi(\phi(\phi(a_1)a_2)a_3)+  \phi(\phi(a_1\phi(a_2))a_3).  
\end{split}
\end{equation}

If we assume, as in the case of operator-valued free probability, that $A$ is a $B$-algebra
(for some unitary associative algebra $B$) and that $\phi$ is a bimodule map, this reduces to
\begin{equation}\label{ex:ncp}
\begin{split}
\kappa_3[a_1a_2a_3]=&\phi(a_1a_2a_3)-\phi(a_1)\phi(a_2a_3)-\phi(a_1\phi(a_2)a_3) - \phi(a_1a_2)\phi(a_3)\\
&+2  \phi(a_1a_2a_3)
\end{split}  
\end{equation}

If moreover $\phi$ is scalar valued, we can rewrite this evaluation as
\begin{equation} \label{eq:mixedcumulantsdef}
\kappa_n[a_1,\dots,a_n] = 
\sum_{t\in \PST_n} (-1)^{i(t)-1} \prod_{v\in
\ii(T)} \phi\left( \prod_{v \measuredangle i}  a_i  \right)
\end{equation}
where in the latter product, $v \measuredangle i$ means that the internal vertex
$v$ has a clear view to the $i$th sector between the $i$th and $(i+1)th$ leaves.
For example, the above tree 
gives the term $\phi(a_1)\phi(a_2)\phi(a_3a_6)\phi(a_4a_5)$.

Let us show that this expression is equivalent to Speicher's formula, so that\\
 $\kappa_n[a_1,\ldots,a_n]$
is indeed the value of the free cumulant.
For this, we need a natural map from 
prime Schr\"oder trees to noncrossing partitions.
Such a map can be defined in terms for noncrossing arrangements of binary trees:

\begin{definition}
 A noncrossing arrangement of binary trees is a set of binary trees, whose leaves are labeled 
 with integers from $1$ to $n$, in such a way that the canonical drawing of the trees does not create any crossing.
 Let $\mathcal{A}_n$ denote the set of such objects.
\end{definition}

\begin{proposition} \label{bijpsttrees}
 There is a bijection between ${\rm PST}_n$ and $\mathcal{A}_n$, which can be defined as follows.
 Let $t \in {\rm PST}_n$, then its image is obtained by:
 \begin{itemize}
  \item removing each middle edge (i.e. an edge which is not leftmost or rightmost among all edges below some 
        internal vertex)
  \item removing the root and all edges below.
  \end{itemize}
\end{proposition}

See Figure~\ref{arrtrees} for an example.

\begin{figure}[h!tp]
 \tikzset{every picture/.style={scale=0.4}}
 \begin{tikzpicture}
   \tikzstyle{ver} = [circle, draw, fill, inner sep=0.5mm]
   \tikzstyle{edg} = [line width=0.6mm]
   \node[ver] at (1,0) {};
   \node[ver] at (2,0) {};
   \node[ver] at (3,0) {};
   \node[ver] at (4,0) {};
   \node[ver] at (5,0) {};
   \node[ver] at (6,0) {};
   \node[ver] at (7,0) {};
   \node[ver] at (8,0) {};
   \node[ver] at (9,0) {};
   \node[ver] at (10,0) {};
   \node[ver] at (11,0) {};
   \node      at (1.5,-1) {1};
   \node      at (2.5,-1) {2};
   \node      at (3.5,-1) {3};
   \node      at (4.5,-1) {4};
   \node      at (5.5,-1) {5};
   \node      at (6.5,-1) {6};
   \node      at (7.5,-1) {7};
   \node      at (8.5,-1) {8};
   \node      at (9.5,-1) {9};
   \node      at (10.5,-1) {10};
   \draw[edg] (1,0) to (3.5,2.5);
   \draw[edg] (6,0) to (3.5,2.5);   
   \draw[edg] (3,2) to (5,0);   
   \draw[edg] (4.5,0.5) to (4,0);   
   \draw[edg] (2,0) to (2.5,0.5);   
   \draw[edg] (3,0) to (2.5,0.5);      
   \draw[edg] (7,0) to (8.5,1.5);      
   \draw[edg] (10,0) to (8.5,1.5);
   \draw[edg] (9,1) to (8,0);
   \draw[edg] (2.5,0.5) to (3,2);   
   \draw[edg] (9,0) to (9,1);   
   \draw[edg] (7,3.5) to (3.5,2.5);   
   \draw[edg] (7,3.5) to (8.5,1.5);   
   \draw[edg] (7,3.5) to (11,0);   
 \end{tikzpicture}
 \hspace{2.5cm}
 \begin{tikzpicture}
   \tikzstyle{ver} = [circle, draw, fill, inner sep=0.5mm]
   \tikzstyle{edg} = [line width=0.6mm]
   \node[ver] at (1,0) {};
   \node[ver] at (2,0) {};
   \node[ver] at (3,0) {};
   \node[ver] at (4,0) {};
   \node[ver] at (5,0) {};
   \node[ver] at (6,0) {};
   \node[ver] at (7,0) {};
   \node[ver] at (8,0) {};
   \node[ver] at (9,0) {};
   \node[ver] at (10,0) {};
   \node      at (1,-1) {1};
   \node      at (2,-1) {2};
   \node      at (3,-1) {3};
   \node      at (4,-1) {4};
   \node      at (5,-1) {5};
   \node      at (6,-1) {6};
   \node      at (7,-1) {7};
   \node      at (8,-1) {8};
   \node      at (9,-1) {9};
   \node      at (10,-1) {10};
   \draw[edg] (1,0) to (3.5,2.5);
   \draw[edg] (6,0) to (3.5,2.5);   
   \draw[edg] (3,2) to (5,0);   
   \draw[edg] (4.5,0.5) to (4,0);   
   \draw[edg] (2,0) to (2.5,0.5);   
   \draw[edg] (3,0) to (2.5,0.5);      
   \draw[edg] (7,0) to (8.5,1.5);      
   \draw[edg] (10,0) to (8.5,1.5);
   \draw[edg] (9,1) to (8,0);
 \end{tikzpicture}
 \caption{The bijection from prime Schr\"oder trees to noncrossing arrangements of binary trees. \label{arrtrees}}
\end{figure}
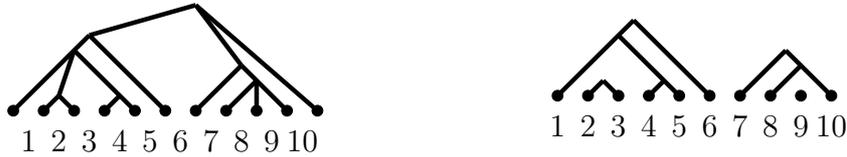

\begin{note}{\rm
Note that in a prime Schr\"oder tree, we label the sectors, whereas in the arrangement of trees, 
we label the leaves. This  means that through the bijection, we need to shift each label in a sector 
to the leaf to its left.
}
\end{note}

To avoid unnecessary notation, we just write $t\mapsto A$ if
the bijection defined in the previous proposition sends $t\in {\rm PST}_n$ to $A\in\mathcal{A}_n$.
Note that there is one less internal vertex in $A$ than
in $t$. Since 
the number of internal vertices of a binary tree is its number of leaves minus one,
\begin{equation} \label{signinternalvertices}
  (-1)^{i(t)-1}  =  (-1)^{ \sum_{a\in A} i(a) } = (-1)^{ n- \# A },
\end{equation}
where $\# A$ denotes the cardinality of $A$.

Also, $A\in\mathcal{A}_n$ can be sent to a noncrossing partition
$\pi$ as follows, which we also denote $A\mapsto \pi$.
It is defined by the condition that two integers are in the same block of $\pi$ if they are in the 
same tree in $A$. For example, the noncrossing arrangement in Figure~\ref{arrtrees} gives the 
noncrossing partition $1456|23|78{\rm A}|9$.
The map $A\mapsto \pi$ is not a bijection, but for $\pi\in {\rm NC}_n$, we have:
\begin{equation}  \label{atopi}
   \#\{ A\in\mathcal{A}_n \, : \, A\mapsto \pi  \} = \prod_{B\in \pi} C_{\# B -1 },
\end{equation}
where $C_n$ denote the Catalan numbers.
Indeed, the construction of  $A$ amounts to choosing a binary tree with $\# B$ leaves for each block $B\in\pi$,
whence the product of Catalan numbers.

\begin{definition}
 The {\it Kreweras complement} of a noncrossing partition $\pi \in {\rm NC}_n$ is
 the noncrossing partition $\pi^c$ defined by the following process:
 \begin{itemize}
  \item we draw $2n$ dots representing integers and primed integers $1,1',2,2',\dots,n,n'$ in this order,
  \item $\pi$ is drawn as a noncrossing partition on $1,\dots,n$ in the usual way,
  \item $\pi^c$ is the coarsest noncrossing partition on $1',\dots,n'$ that can be drawn without crossing
        $\pi$.
 \end{itemize}
Then, $\pi^c$ is identified with an element of ${\rm NC}_n$ by removing the prime symbols.
\end{definition}

For example, Figure~\ref{kreweras} shows that if $\pi= 134|2|57|6|8$, then $\pi^c = 12|3|478|56$.
Practically, there is a convenient equivalent definition. If we have a set of noncrossing paths above the horizontal axis
linking pairs of integers (as in Figure~\ref{kreweras}), it defines a noncrossing partition $\pi$ by taking 
connected components, and $i,j$ are in a same block of $\pi^c$ iff there exists a path from $i$ to $j$ that
does not cross the other paths.

\tikzset{every picture/.style={scale=0.6}}
\begin{figure}[h!tp]
 \begin{tikzpicture}
   \tikzstyle{edg} = [line width=0.6mm]
   \tikzstyle{edh} = [line width=0.6mm,dashed]
   \tikzstyle{ver} = [circle, draw, fill, inner sep=0.5mm]
   \node[ver] at (1,0) {};
   \node[ver] at (2,0) {};
   \node[ver] at (3,0) {};
   \node[ver] at (4,0) {};
   \node[ver] at (5,0) {};
   \node[ver] at (6,0) {};
   \node[ver] at (7,0) {};
   \node[ver] at (8,0) {};
   \node[ver] at (9,0) {};
   \node[ver] at (10,0) {};
   \node[ver] at (11,0) {};
   \node[ver] at (12,0) {};
   \node[ver] at (13,0) {};
   \node[ver] at (14,0) {};
   \node[ver] at (15,0) {};
   \node[ver] at (16,0) {};
   \node      at (1,-1) {1};
   \node      at (2,-1) {1'};
   \node      at (3,-1) {2};
   \node      at (4,-1) {2'};
   \node      at (5,-1) {3};
   \node      at (6,-1) {3'};
   \node      at (7,-1) {4};
   \node      at (8,-1) {4'};
   \node      at (9,-1) {5};
   \node      at (10,-1) {5'};
   \node      at (11,-1) {6};
   \node      at (12,-1) {6'};
   \node      at (13,-1) {7};
   \node      at (14,-1) {7'};
   \node      at (15,-1) {8};
   \node      at (16,-1) {8'};
   \draw[edg] (1,0) to[bend left=60] (5,0);
   \draw[edg] (5,0) to[bend left=60] (7,0);
   \draw[edg] (9,0) to[bend left=60] (13,0);
   \draw[edh] (2,0) to[bend left=60] (4,0);
   \draw[edh] (8,0) to[bend left=60] (14,0);
   \draw[edh] (10,0) to[bend left=60] (12,0);
   \draw[edh] (14,0) to[bend left=60] (16,0);
 \end{tikzpicture}
 \caption{The Kreweras complement of a noncrossing partition. \label{kreweras}}
\end{figure}
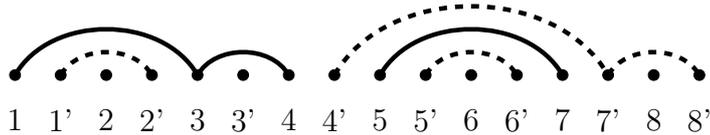

\begin{lemma} \label{treenc}
 Suppose that $t\mapsto A \mapsto \pi$ (with the previous notation). Then the noncrossing partition
 $\pi^c$ is obtained from $t$ by the following condition: $i$ and $j$ are in the same block of $\pi^c$
 if and only if some internal vertex of $t$ has a clear view to the $i$th and $j$th sectors.
\end{lemma}

See for example Figure~\ref{treencfig} where $\pi = 12|3|46|5$ and $\pi^c = 1|236|45 $.

\begin{proof}
We place labels $1,2,\dots,n+1$ at the leaves of $t$, and labels $1',2',\dots,n'$ in sectors of $t$, so that all labels 
are ordered $1,1',2,2',\dots$ as in the definition of the Kreweras complement. See Figure~\ref{treencfig}.

Suppose first that $v\in t$ is an internal vertex, and $v\measuredangle i'$, $v\measuredangle j'$. Then it is possible to 
draw a path from sector $i'$ to sector $j'$ crossing only middle edges below $v$. This path does not cross the
arrangement of binary trees obtained from $t$ by the bijection of Proposition~\ref{bijpsttrees}, since we remove middle
edges, and it follows that $i$ and $j$ are in a same block of $\pi^c$.

Suppose then that $i$ and $j$ are in a some block of $\pi^c$, {\it i.e.}, there is a path from $i'$ to $j'$ that
does not cross the edges in $A$. So it crosses only middle edges of $t$. These middle edges are all below some vertex $v$,
because each sector can be connected to a unique internal vertex of $A$ by a path that does not cross edges of $A$ (except 
the case where these sectors have no internal vertex above, then the path crosses only middles starting from the
root of $t$).
\end{proof}

\begin{figure}[h!tp]
 \begin{tikzpicture}
   \tikzstyle{ver} = [circle, draw, fill, inner sep=0.5mm]
   \tikzstyle{edg} = [line width=0.6mm]
   \tikzstyle{edh} = [line width=0.6mm,dashed]
   \node      at (1,-0.8) {1};
   \node      at (2,-0.8) {1'};
   \node      at (3,-0.8) {2};
   \node      at (4,-0.8) {2'};
   \node      at (5,-0.8) {3};
   \node      at (6,-0.8) {3'};
   \node      at (7,-0.8) {4};
   \node      at (8,-0.8) {4'};
   \node      at (9,-0.8) {5};
   \node      at (10,-0.8) {5'};
   \node      at (11,-0.8) {6};
   \node      at (12,-0.8) {6'};
   \node      at (13,-0.8) {7};
   \draw[edg] (1,0) to (2,1);
   \draw[edg] (3,0) to (2,1);
   \draw[edg] (7,0) to (9,1);
   \draw[edg] (9,0) to (9,1);
   \draw[edg] (11,0) to (9,1);
   \draw[edg] (2,1) to (8,4);
   \draw[edg] (5,0) to (8,4);
   \draw[edg] (5,0) to (8,4);
   \draw[edg] (9,1) to (8,4);
   \draw[edg] (13,0) to (8,4);
   \draw[edh] (4,0) to[bend left=60] (6,0);
   \draw[edh] (6,0) to[bend left=60] (12,0);
   \draw[edh] (8,0) to[bend left=60] (10,0);
   \node[ver] at (2,0) {};
   \node[ver] at (4,0) {};
   \node[ver] at (6,0) {};
   \node[ver] at (8,0) {};
   \node[ver] at (10,0) {};
   \node[ver] at (12,0) {};
 \end{tikzpicture}
 \caption{The proof of Lemma~\ref{treenc}.\label{treencfig}}
\end{figure}
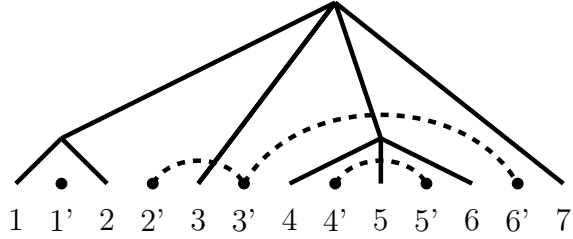

A consequence of this proposition is that the product 
$\prod_{v \measuredangle i}  a_i$ for an internal vertex $v$ of $t$, is equal to the product
$\prod_{i\in B}  a_i$ where $B$ is the appropriate block of $\pi^c$. 
So we can rewrite Equation~\eqref{eq:mixedcumulantsdef} in terms of arrangements 
of binary trees, using Equation~\eqref{signinternalvertices} for the sign:
\begin{equation} 
  \kappa[a_1,\dots,a_n] = \sum_{ A \in \mathcal{A}_n } (-1)^{n-\# A} 
      \phi_{\pi^c} [a_1,\dots,a_n].
\end{equation}
Note that the sum is over $A$ and that  $\pi$ depends on $A$ by $A\mapsto \pi$.
We can rewrite this in terms of noncrossing partitions, using Equation~\eqref{atopi}
to get the right multiplicities, which yields
\begin{equation}  \label{kappanc1}
  \kappa[a_1,\dots,a_n] = \sum_{ \pi \in {\rm NC}_n } (-1)^{n-\# \pi}
     \bigg( \prod_{B\in \pi} C_{\# B -1 }  \bigg)
  \phi_{\pi^c } [a_1,\dots,a_n].
\end{equation}

The product of Catalan numbers in this expression can be identified as a value of the 
M\"obius function of ${\rm NC}_n$, as shown by Kreweras \cite{kreweras}:
\[
  \mu(\hat 0, \pi ) = (-1)^{ n- \# \pi }  \prod_{ B \in \pi }  C_{ \#B - 1 }.
\]
We also need some further properties of the map $\pi\mapsto\pi^c$.
It is an  anti-automorphism of the poset ${\rm NC}_n$, {\it i.e.},  an order-reversing bijection. 
This shows that the interval $[\hat 0, \pi]$ is anti-isomorphic to the interval $[ \pi^c , \hat 1 ]$.
Moreover, each interval in the noncrossing partition lattice is isomorphic to a product
of noncrossing partition lattices of smaller orders, so  that each interval is a self-dual poset.
Thus, the anti-isomorphism between $[\hat 0, \pi]$ and $[ \pi^c , \hat 1 ]$ becomes an isomorphism,
when composed with some anti-automorphism.
This shows the equality of the M\"obius functions
\begin{equation}
   \mu(\hat 0, \pi ) = \mu(\pi^c,\hat 1).
\end{equation}
Back to Equation~\eqref{kappanc1}, we obtain:
\begin{equation}
   \kappa[a_1,\dots,a_n] = \sum_{\pi \in {\rm NC}_n}  \mu(\pi^c, \hat 1) \phi_{\pi^c}[a_1,\dots,a_n].
\end{equation}
We can replace $\pi^c$ by $\pi$ in the summand,  and this gives Speicher's formula.
Note that our argument remains valid for an operator valued probability, if we interpret
$\phi_\pi$ as $\hat\phi[\pi]$ in \cite[Definition 2.1.1]{Spei2} (cf. Eq. \eqref{ex:ncp}). We can also
observe that when $\phi$ is a bimodule map, any evaluation of a Schr\"oder tree can be reduced to that of
a tree in which the leftmost subtree of any internal node is a leaf. These trees, which
may be called right-directed Schr\"oder trees, are counted by the Catalan numbers,
and Speicher's definition of $\hat\phi[\pi]$ induces a particular bijection between these and non-crossing partitions.

For example, the term
\begin{equation}
\phi(a_1\phi(\phi(a_2)a_3\phi(a_4)a_5\phi(a_6))a_7)
\end{equation}
can be rewritten as the right-directed tree
\begin{equation}
\phi(a_1\phi(a_2\phi(a_3\phi(a_4)a_5\phi(a_6)))a_7)
\end{equation}
and corresponds to the non-crossing partition $\pi=15|2|35|6$.

\section{The cluster property}

Two subsets $B,C\subset A$ have the {\it cluster property} if
for any $b_1,\dots,b_j\in B$ and $c_1,\dots,c_k\in C$, we have
$\phi(b_1\cdots b_j c_1 \cdots c_k ) = \phi(b_1\cdots b_j ) \phi(c_1 \cdots
c_k ) $.
Neu and Speicher \cite{neuspeicher} have shown that this property is
characterized
by a vanishing condition on free cumulants, and this was also obtained by
Ebrahimi-Fard and Patras \cite{EFP1} with the algebraic definition of free
cumulants in terms of infinitesimal characters. We show that this property also follows
from our expression in terms of prime Schr\"oder trees, using a simple
combinatorial argument. Our proof has the advantage of being nonrecursive.

\begin{proposition}
If two subsets $B,C$ of $A$ have the cluster property, then for any $j>0$,
$k>0$,
and $b_1,\dots,b_j\in B$ and $c_1,\dots,c_k \in C$ we have
\[
  \kappa(b_1,\dots,b_j,c_1,\dots,c_k)=0.
\]
\end{proposition}

\begin{proof}
From Equation~\eqref{eq:mixedcumulantsdef}, we have:
\begin{equation} \label{kbbbccc}
\kappa(b_1,\dots,b_j,c_1,\dots,c_k) = \sum_{T\in \PST_n} (-1)^{i(T)-1}
\prod_{v\in \ii(T)} \phi\left( \prod_{v \measuredangle i} b_i \prod_{v
\measuredangle \ell} c_\ell \right)
\end{equation}
where the indices $i$ and $\ell$ are restricted to $1\leq i \leq j$ and
 $1\leq \ell \leq k$,
and $n=j+k$.
We will give a fixed point-free involution on the set $\PST_n$, showing
that
the terms in this sum cancel pairwise, so that it vanishes.

Let $t\in \PST_n$ and label its $n$ sectors with
$b_1,\dots,b_j,c_1,\dots,c_k$.
Consider the path starting from the $(j+1)$th leaf of the tree ({\it i.e.},
the leaf between $b_j$ and $c_1$) up to the root.
(We draw this path in bold in the pictures).

Suppose first that this path goes through a middle edge of $t$, say the
$i$th
one among $p$ edges below some vertex $v$.
Then we can perform the following local move to get a new tree $u$:
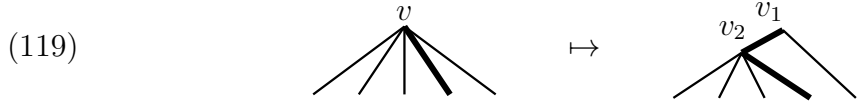
\begin{equation} \label{localmove}
  \begin{tikzpicture}
   \tikzstyle{edg} = [line width=0.3mm]
   \tikzstyle{edgg} = [line width=0.8mm]
   \node at (2,1.8) {$v$};
   \draw[edg] (0,0) to (2,1.5);
   \draw[edg] (1,0) to (2,1.5);
   \draw[edg] (2,0) to (2,1.5);
   \draw[edgg] (3,0) to (2,1.5);
   \draw[edg] (4,0) to (2,1.5);
  \end{tikzpicture}
  \qquad \mapsto \qquad
  \begin{tikzpicture}
   \tikzstyle{edg} = [line width=0.3mm]
   \tikzstyle{edgg} = [line width=0.8mm]
   \node at (1.3,1.5) {$v_2$};
   \node at (2.1,1.9) {$v_1$};
   \draw[edg] (0,0) to (1.5,1);
   \draw[edg] (1,0) to (1.5,1);
   \draw[edg] (2,0) to (1.5,1);
   \draw[edgg] (3,0) to (1.5,1);
   \draw[edgg] (1.5,1) to (2.4,1.5);
   \draw[edg] (4,0) to (2.4,1.5);
  \end{tikzpicture}
\end{equation}

More formally, the corolla formed by $v$ and the $p$ edges below is
transformed
into a corolla on a vertex $v_1$ with $p-i+1$ edges below it, and to the
first one of
these edges is attached another corolla on a vertex $v_2$ with $i$ edges
($i=4$ and $p=5$
in the picture).

Denote by
$b_{\alpha_1},\dots,b_{\alpha_{i-1}},c_{\beta_1},\dots,c_{\beta_{k-i}}$ the
$k$ sectors seen from $v$.
So $b_{\alpha_1},\dots,b_{\alpha_{i-1}}$ are the sectors seen from $v_2$
and
$c_{\beta_1},\dots,c_{\beta_{k-i}}$ are those seen from $v_1$ in $u$.

In the term indexed by $t$ in the right-hand side of \eqref{kbbbccc}, we
have a factor
\begin{equation}
\phi(b_{\alpha_1},\dots,b_{\alpha_{i-1}},c_{\beta_1},\dots,c_{\beta_{k-i}}),
\end{equation} 
whereas in the term
indexed by $u$, we have instead the two factors
$\phi(b_{\alpha_1},\dots,b_{\alpha_{i-1}})$ and
$\phi(c_{\beta_1},\dots,c_{\beta_{k-i}} )$.
By the cluster property, these two terms are equal up to a sign.
Since $u$ has one more internal vertex than $t$, it contributes to the sum
in \eqref{kbbbccc}
with an opposite sign. So the two terms indexed by $t$ and $u$ cancel each
other.

It remains to see how this local move can be used to define the fixed point
free
involution. Let $t\in\PST_n$, and draw as before a path from the $(j+1)$th
leaf
up to the root. Let us follow this path from bottom to top, and stop
when finding:
\begin{itemize}
 \item either middle edge (case 1),
 \item or a right edge, followed by a left edge just above it (case 2).
\end{itemize}
Note that these two cases correspond to the two sides of
\eqref{localmove},
so that we can define the involution by performing the local move going from
one case to the other.

To see that this map is well-defined, it only remains to see that the two
cases are
exhaustive. Let $t\in \PST_n$ such that we are not in case 1, {\it i.e.}
the path
does not cross any middle edge.
Since the tree is prime, the right edge starting from the root arrives at
the rightmost leaf
(which does not separate two sectors), so the path arrives to the root by
the left edge.
Also, the path contains at least a right edge (otherwise, it would connect
the root to the
leftmost edge, which does not separate two sectors).
It follows that we can find two edges as in the right part of
\eqref{localmove},
{\it i.e.}, we are indeed in case 2.
\end{proof}


\section{The Hopf algebra of decorated Schr\"oder trees}

\subsection{}
Let $A$ be any set (decorations), and $T(A)=\K \otimes (\otimes_{n\geq 1} T_n(A))$ 
the free associative $\K$-algebra over $A$, regarded as the tensor algebra 
of the linear span of $A$.

Using the grading of $H_\SS$, we can define a decorated version of the algebra $H_\SS$ 

\begin{equation}
H_\SS(A)=\K \oplus \bigoplus_{n\geq 1} (H_{\SS,n}\otimes T_n(A)).
\end{equation}

This space has an obvious algebra structure, and  it is also easy to extend the Hopf algebra structure of $H_\SS$. 
Consider a tree $T\in H_{\SS,n} $ and $w=a_1\dots a_n$ in $T_n(A)$. Since $T$ has $n+1$ leaves, one can label its sectors
from left to right with $a_1,\dots,a_n$ and identify $T\otimes w$ with this decorated tree. For instance
\begin{equation}
\tikzset{every picture/.style={scale=0.6}}
  \begin{tikzpicture}
      \tikzstyle{ver} = [circle, draw, fill, inner sep=0.5mm]
      \tikzstyle{edg} = [line width=0.6mm]
      \node      at (1.5,-0.6) {$a_1$};
      \node      at (2.5,-0.6) {$a_2$};
      \node      at (3.5,-0.6) {$a_3$};
      \node      at (4.5,-0.6) {$a_4$};
      \node      at (5.5,-0.6) {$a_5$};
      \node      at (6.5,-0.6) {$a_6$};
      \draw[edg] (2,0) to (2.5,0.5);
      \draw[edg] (3,0) to (2.5,0.5);
      \draw[edg] (2.5,0.5) to (1.8,1);
      \draw[edg] (1,0) to (1.8,1);
      \draw[edg] (4,0) to (5,1);
      \draw[edg] (5,0) to (5,1);
      \draw[edg] (6,0) to (5,1);
      \draw[edg] (1.8,1) to (4.5,2);
      \draw[edg] (5,1) to (4.5,2);
      \draw[edg] (7,0) to (4.5,2);
  \end{tikzpicture}
\end{equation} 

In an admissible cut $c$ for such a tree, $P^c(T)$ obviously inherits the letters $a_i$ associated with the subtrees in $P^c(T)$,
 and $R^c(T)$ keeps the letters which can be viewed from the internal vertices of $T$ that are still in $R^c(T)$.  
It is clear that $H_\SS(A)$ is a Hopf algebra, and a straightforward adaptation  of the proof of Theorem \ref{th:unshundec} shows that 
\begin{theorem}
 $H_\SS(A)$ is a codendriform bialgebra.
\end{theorem}
The decorated analog of Theorem \ref{th:kap} reads on characters
\begin{theorem}\label{th:kapdec}
Let $\phi$ be a linear form on $T(A)$. Extend it to a map
$\phi:\ H_\SS(A)\rightarrow \C$ sending the decorated corollas to $\phi(w)$ where $w$ is the decorating word and the other trees to 0
(regarded as an infinitesimal character of $\H_\SS(A)$), and let $\Phi$
be its extension to a character of  $\H_\SS(A)$. Then,
\begin{equation}\label{eq:kapdec}
\Phi = \epsilon + \kappa\prec \Phi
\end{equation}
where $\kappa$ is the infinitesimal character on $H_\SS(A)$ defined by
\begin{equation}
\kappa(T\otimes a_1 \dots a_n) = \begin{cases}
(-1)^{i(t)-1} \displaystyle\prod_{v\in
{\rm int}(T)} \phi\left( \prod_{v \measuredangle i}  a_i  \right) & if \ T \in \PST \\
0 & otherwise
\end{cases}
\end{equation}
where $v \measuredangle i$ means that the internal vertex
$v$ has a clear view to the $i$th sector between the $i$th and $(i+1)th$ leaves.
\end{theorem}

\Proof As in \cite{EFP1}, $\Phi$ is necessarily a character. 
If for a decorated tree $T\otimes a_1 \dots a_n$, $T$ is not in $\PST$, the r.h.s of (\ref{eq:kapdec}) is necessarily 0. 
Otherwise, the right-hand side of (\ref{eq:kapdec}) reads as the binomial expansion of  
\begin{equation}
(1-1)^n \prod_{v\in
\ii(T)} \phi\left( \prod_{v \measuredangle i}  a_i  \right)
\end{equation}
where
$n$ is the number of its internal nodes whose descendants are leaves.
Thus, only the  corollas survive, with value $\phi(a_1\dots a_n)$. \qed

When summing over the different prime Schr\"oder trees decorated by the same word, one recovers Speicher's formula 
\begin{equation}
\kappa_n[a_1,\dots,a_n] = 
\sum_{t\in \PST_n} (-1)^{i(t)-1} \prod_{v\in
\ii(T)} \phi\left( \prod_{v \measuredangle i}  a_i  \right).
\end{equation}
Thus, in addition to the combinatorial argument of Section \ref{sec:spei},
we see that this expression can be recovered algebraically, with the help of morphisms of codendriform bialgebras.

\subsection{The formula of Ebrahimi-Fard and Patras}

In \cite{EFP1} and \cite{EFP},
free cumulants appear as the solution of a dendriform equation for characters of $T(T_{\geq 1}(A))$. 
The latter algebra is the free associative algebra over words $w$ in $T_{\geq 1}(A)$, so that its elements can
be viewed as linear combinations of segmented words $w_1 | \dots | w_s$.   
Let us just recall some results of \cite{EFP,EFP1}, 
the definition of an codendriform bialgebra having already been recalled in Theorem \ref{th:unshundec}. 

The algebra $T(T_{\geq 1}(A))$ is endowed with a coproduct  defined on words as follows. 
Consider a word $w=a_1 \dots a_n$ as a ladder tree whose vertices are decorated from the root to the leaf by $a_1,\dots,a_n$. 
For any subset $S=\lbrace i_1<\dots<i_k \rbrace$ of $\lbrace 1, \dots , n\rbrace $, 
denote by $R^S(w)=a_{i_1}\dots a_{i_k}=a_S$  the ladder labeled with the letters $a_i$ corresponding to the subscripts in $S$.

Once these vertices are removed from the original ladder, there remains some ladders which define a 
segmented word $P^S(w)$,
the  tensor (bar) product of these connected components. 
That is, $P^S(w)=w_S^1 | \dots|w_S^l $ where, in each $w_S^j$, the subscripts of the letters in each factor are consecutive integers,
but the union of the subscripts of two consecutive  factors is not an interval. 

It is proved in \cite{EFP1} that with the coproduct defined on words by 
\begin{equation}
\Delta(w)=\Delta (a_1\dots a_n)=\sum_{S\subset \lbrace 1, \dots , n\rbrace}   R^S(w)\otimes P^S(w),
\end{equation}
$T(T_{\geq 1}(A))$ is a Hopf algebra which is also a codendriform algebra for the splitting
of the coproduct
\begin{equation}
\Delta = \Delta^+_\succ + \Delta^+_\prec
\end{equation}
defined on words $w=a_1\dots a_n$ in $T_{\geq 1}(A)$ by
\begin{equation}
\Delta^+_\prec(a_1\dots a_n) = \sum_{S\subset \lbrace 1, \dots , n\rbrace \ ,\ n\in S} R^S(w)\otimes P^S(w).
\end{equation}

The relation between free moments and free cumulants is then given by the same equation as in Theorem \ref{th:kapdec}, 
relating a linear map $\phi$ defined on words of $T_{\geq 1}(A)$ and its extension $\tilde\Phi$ as a character of $T(T_{\geq 1}(A))$
to the free cumulants (defining an infinitesimal character $\tilde{\kappa}$ on  $T(T_{\geq 1}(A))$) 
by the same equation as in Theorem \ref{th:kapdec}. This also proves  Speicher's formula,
which can thus be interpreted in terms of morphisms  of codendriform bialgebras. 

Note that in \cite{EFP1}, the codendriform bialgebra structure is defined by splitting the coproduct
according to whether 
 $1$ is in $S$ instead of  $n$ for us. 
Both structures are obviously isomorphic under reversal of the words 
$a_1\dots a_n \in T_{\geq 1}(A)\mapsto a_n\dots a_1$.

%
%

\subsection{ A codendriform Hopf morphism}


\begin{theorem}
Let $\iota$ be the algebra morphism $T(T_{\geq 1}(A)) \rightarrow H_\SS(A)$ sending a
word $w=a_1\cdots a_n$ to the sum of all trees with $n$ sectors decorated from
left to right by $a_1,\ldots,a_n$. Then,\\
(i) $\iota$ is a coalgebra morphism;\\
(ii) $\iota$ is a codendriform morphism: 
\begin{equation}
(\iota\otimes\iota) \circ \Delta_\prec(w) = \Delta_\prec\circ\iota(w).
\end{equation}
\end{theorem}

This result implies the formula for free cumulants, as $\iota$ send the maps $\Phi$ and $\kappa$ of Theorem \ref{th:kapdec} 
to the free moments $\tilde{\Phi}=\phi\circ \iota$ and cumulants $\tilde{\kappa}=\kappa \circ \iota$ of \cite{EFP1}. 

\begin{proof} Consider a word $w=a_1 \dots a_n$, $n\geq 1$. We can assume without loss of generality that the decorations are pairwise distinct. We must first prove that
\begin{equation}
\Delta \circ \iota(w)=(\iota \otimes \iota)\circ \Delta (w).
\end{equation}

For any subset $S$ of $\{ 1,\dots, n\}$, there is a unique term 
\begin{equation}
R^S(w)\otimes P^S(w)=a_S\otimes w_S^1 | \dots|w_S^l 
\end{equation}
corresponding to $S$ in $\Delta(w)$. 
We shall say that a sequence of trees $(T_0,T_1,\dots, T_k)$ is compatible with $S$ (and write $S\sim (T_0,T_1,\dots, T_k)$) 
if $k=l$, $\wt(T_0)=|S|$ and for all $1\leq i\leq l$, $\wt(T_i)=l(w_S^i)$. 
By definition of $\iota$,  the set
\begin{equation}
A(w)=\left\lbrace(T_0; T_1\dots T_k ; R^S(w) ; P^S(w) ), \ S\subset\{1,\dots,n\} ;\ S\sim(T_0,T_1,\dots, T_k) \right\rbrace
\end{equation}
is such that 
\begin{equation}
(\iota \otimes \iota)\circ \Delta (w)=\sum_{(T_0,T_1\dots T_k ; w^0 ;w^1 |\dots |w^k)\in A(w)} (T_0\otimes w^0) \otimes \prod (T_i \otimes w^i).
\end{equation}

Let $c$ be an admissible cut of a tree
$T\in \PT_{n+1}$. 
For the decorated tree $T\otimes a_1...a_n$ the decorations of $R^c(T)$ correspond to a unique set 
$S(c)=\lbrace i_1<\dots<i_k \rbrace$ of $\lbrace 1, \dots , n\rbrace $,
so that in the coproduct, we find $R^c(T)\otimes R^{S(c)}(w)=R^c(T)\otimes a_{S(c)}$ on the left-hand side of 
the term corresponding to this cut. 
On the right-hand side, if $P^c(T)=T_1\dots T_k$, the corresponding decorations are 
the connected components of $P^{S(c)}(w)=w_{S(c)}^1 | \dots|w_{S(c)}^{k}$: 
the subscripts of each $w_{S(c)}^i$ form an interval  whose length is the weight of $T_i$. 
Now,
one can observe that, for any tree $T$, the set 
\begin{equation}
A_T(w)=\left\lbrace (R_c(T) ; P^c(T) ; R^{S(c)}(w) ; P^{S(c)}(w)) \ , c\in \cut(T) \right\rbrace
\end{equation}
is a subset of $A(w)$ such that
\begin{equation}
\Delta (T\otimes w)=\sum_{(T_0,T_1\dots T_k ; w^0 ;w^1 |\dots |w^k)\in A_T(w)} (T_0\otimes w^0) \otimes \prod (T_i \otimes w^i).
\end{equation}
There could be an ambiguity in this formula if there were repeated elements  in the definition of  $A_T(w)$ but this is never  the case. 
For a given tree $T$, there maybe several cuts that give rise to the same undecorated term but the cuts would be at different positions and the induced decorations would be different.

Similarly, if we consider two different trees $T,T' \in \PT_{n+1}$, 
then $A_T(w) \cap A_{T'}(w)=\emptyset$. 
There could be  admissible cuts $c$ and $c'$  giving the same undecorated terms in the coproducts of $T$ and $T'$,
but if the associated decorations were the same, the initial ordering of the subscripts in $a_1 \dots a_n$ 
would impose that $T\otimes w= T'\otimes w$. 
For instance, if the first interleaves of $R^c(T)=R^{c'}(T')$ have the same decorations, 
say $a_1$ and $a_5$, then the first factor of $P^c(T)=P^{c'}(T')$ is decorated by $a_2 a_3 a_4$ 
and the subtree from which one sees $a_1 a_2 a_3 a_4 a_5$ is the same in $T\otimes w$ and $T' \otimes w$. 
Furthermore, due to the order of the letters in $w=a_1  \dots a_n$, any given element in $\cup_{T} A_T(w)$
appears in only one of the subsets $A_T(w)$ and corresponds to a unique admissible cut of $T$.

Thus, the inclusion map from the disjoint union $\cup_{T} A_T(w)$ to $A(w)$ is injective. 

In then same way, one can check that any element of $A(w)$ occurs in a unique $A_T(w)$. 
As before, one can recover the unique tree and the unique cut such that the letters are in the correct order in $T\otimes w$. 
This means that the inclusion map is actually a bijection:
\begin{eqnarray}
(\iota \otimes \iota)\circ \Delta (w)&=& \sum_{(T_0,T_1\dots T_k ; w^0 ;w^1 |\dots |w^k)\in A(w)} (T_0\otimes w^0) \otimes \prod (T_i \otimes w^i) \\
 &=&\sum_T \sum_{(T_0,T_1\dots T_k ; w^0 ;w^1 |\dots |w^k)\in A_T(w)} (T_0\otimes w^0) \otimes \prod (T_i \otimes w^i)  \\
 &=& \sum_T \Delta(T\otimes w) \\
 &=& \Delta \circ \iota (w).
\end{eqnarray}

The proof for the codendriform structure is similar, noticing that, under restriction  to the  $A^\prec(w) \subset A(w)$ 
(resp. $A_T^\prec(w)\subset A_T(w)$) corresponding to the half coproducts, we still have a bijection between $A^\prec(w)$ and the decomposition $\cup_{T} A_T^\prec(w)$ 
associated with the equation
$$
(\iota\otimes\iota) \circ \Delta_\prec(w) = \Delta_\prec\circ\iota(w).
$$
\end{proof}




\section{Miscellaneous remarks}

We have seen that the large Schr\"oder numbers can be expressed as  sums of
products of Catalan numbers:
\begin{equation}
  \# \PST_n = \sum_{\pi \in \NC_n} \prod_{b\in \pi} C_{\#b - 1}.
\end{equation}
This identity appeared already in the paper by Dykema \cite[Corollary~8.4]{Dy},
where it is related to other combinatorial objects called noncrossing linked
partitions.
Dykema's work being also related to operator-valued free probability theory, it is
desirable to look for an explanation of this coincidence. To grasp a better understanding
of what is happening,
note that in each case, the Schr\"oder numbers appear with a multivariate
refinement.
Our  version is obtained when we remove all signs in the expression
of the free cumulants in terms of the moments $(m_i)_{i\geq 1}$ (that we regard as
indeterminates).
The first values are:
\begin{align}
    & m_1, \\
    & m_1^2 + m_2, \\
    & 2m_1^3 + 3m_2m_1 + m_3, \\
    & 5m_1^4 + 6m_2m_1^2 + 2m_2^2 + 8m_3m_1  + m_4.
\end{align}
On another hand, the multivariate refinements of Schr\"oder numbers of
\cite{Dy}
are obtained by expressing moments in terms of the coefficients of the
T-transform,
see Proposition 8.1 there. The first values are 
\begin{align}
    & \alpha_0, \\
    & \alpha_0^2 + \alpha_0\alpha_1, \\
    & \alpha_0^3 + 3 \alpha_0^2 \alpha_1 + \alpha_0\alpha_1^2 + \alpha_0^2
\alpha_2, \\
    & \alpha_0^4 + 6\alpha_0^3 \alpha_1 + 6\alpha_0^2\alpha_1^2 +
4\alpha_0^3\alpha_2 + \alpha_0 \alpha_1^3 + 3 \alpha_0^2\alpha_1\alpha_2 +
\alpha_0^3\alpha_3.
\end{align}
There does not seem to be any clear link
between these families of polynomials.



\footnotesize
\begin{thebibliography}{aa}

%
\bibitem{BFK}{\sc C. Brouder, A. Frabetti} and {\sc C. Krattenthaler},
{\it Non-commutative Hopf algebra of formal diffeomorphisms},
Adv. in Math. {\bf 200} (2006), 479--524.
%
\bibitem{Cha}{\sc F. Chapoton}, {\it
Rooted trees and an exponential-like series}, arXiv:math/0209104.
%
\bibitem{Dy}{\sc K. Dykema}, {\it
Multilinear function series and transforms in free probability theory}, 
Adv. Math. {\bf 208} (2007), 351--407. 
%
%
\bibitem{MEF}{\sc K. Ebrahimi-Fard, D. Manchon}, {\it On
  an extension of Knuth's rotation correspondence to reduced planar
  trees}, J. Noncommut. Geom., 8(2):303--320,
  2014.
%

%
\bibitem{EFLM}{\sc
K. Ebrahimi-Fard, A. Lundervold, D. Manchon},{\it
Noncommutative Bell polynomials, quasideterminants and incidence Hopf algebras}, arXiv:1402.4761.


\bibitem{EFP1}{\sc K. Ebrahimi-Fard} and {\sc F. Patras}, {\it
 Cumulants, free cumulants and half-shuffles}, 
Proc. Royal Soc. London A {\bf 471} (2015), 20140843.


\bibitem{EFP}{\sc K. Ebrahimi-Fard} and {\sc F. Patras}, {\it 
The splitting process in free probability theory},
Internat. Math. Res. Notices 2015; doi: 10.1093/imrn/rnv209.


%
\bibitem{NCSF1}{\sc I.M. Gelfand, D. Krob, A. Lascoux, B. Leclerc,
V.S.  Retakh} and {\sc J.-Y. Thibon},
{\it Noncommutative symmetric functions},
Adv. in Math. {\bf 112} (1995), 218--348.
%
%
\bibitem{JR}{\sc S. A. Joni} and {\sc G.-C. Rota}, {\it
Coalgebras and bialgebras in combinatorics}, 
Contemp. Math. {\bf 6} (1982), 1--47.
%

%
\bibitem{KM}{\sc M. Kapranov} and {\sc Yu. Manin}, {\it
Modules and Morita theorem for operads},
American Journal of Mathematics
{\bf 123} (2001),
811--838.
%
\bibitem{kreweras}{\sc G. Kreweras}, {\it
Sur les partitions non croisées d'un cycle}, Discrete Math.{\bf 1} (1972)1,  333--350.
%
\bibitem{Mcd}{\sc I.G. Macdonald}, {\it Symmetric functions and Hall
    polynomials}, 2nd ed., Oxford University Press, 1995.
%
\bibitem{MNT}{\sc F. Menous, J.-C. Novelli} and {\sc J.-Y. Thibon}{\it
Combinatorics of Poincar\'e's and Schr\"oder's equations}, arXiv:1506.08107.
\bibitem{neuspeicher}{\sc P. Neu} and {\sc R. Speicher}, {\it
A self-consistent master equation and a new kind of cumulants},
Z. Phys. B {\bf 92} (1993), 399--407.
%
\bibitem{NicaSpeicher}
{\sc A. Nica and R. Speicher},
{\it Lectures on the combinatorics of free probability}, 
London Mathematical Society Lecture Note Series, 
Vol. 335, Cambridge University Press, Cambridge, 2006.
%
\bibitem{NTPark}{\sc J.-C. Novelli} and {\sc  J.-Y. Thibon},
{\it Hopf algebras and dendriform structures arising from parking functions},
Fund. Math.   {\bf 193}  (2007),   189--241.
%
\bibitem{NTLag}{\sc J.-C Novelli} and {\sc J.-Y. Thibon},
{\it Noncommutative symmetric functions and Lagrange inversion},
Adv. Appl. Math. {\bf 40} (2008), 8--35.
%
\bibitem{NTDup}{\sc J.-C. Novelli} and {\sc J.-Y. Thibon},
{\it Duplicial algebras and Lagrange inversion}, arXiv:1209.5959.

%
%
\bibitem{Spei1}{\sc R. Speicher},
{\it Multiplicative functions on the lattice on non-crossing partitions and
free convolution},
Math. Annalen {\bf 298} (1994), 141--159.
%
\bibitem{Spei2}{\sc R. Speicher},
{\it Combinatorial theory of the free product with amalgamation and operator-valued free probability theory},
Memoirs of the AMS {\bf 627} (1998). 

%
%
%
\end{thebibliography}
\end{document}